\documentclass[11pt,a4paper]{article}
\pdfoutput=1 
\usepackage[margin=1in]{geometry}
\usepackage{amsmath,amsthm,amssymb,enumerate, bbm ,graphicx,color,caption,upgreek, float, tikz, subcaption,booktabs,longtable, appendix,graphics, pdfpages,rotating,mathtools, array, bm, blkarray,setspace,textcomp,tabularx,pgfplots,ytableau}
\definecolor{firebrick}{rgb}{0.7, 0.13, 0.13}

\usepackage[pdftex]{hyperref}

\ytableausetup{boxsize=1em}

\definecolor{blauw}{RGB}{61,158,255}
\definecolor{donkerblauw}{RGB}{0,0,255}
\definecolor{donkergroen}{RGB}{46,148,0}
\definecolor{donkerrood}{RGB}{204,0,0}

\pgfplotsset{compat=1.14}
\makeatletter 
\newcommand\mynobreakpar{\par\nobreak\@afterheading} 
\makeatother
\makeatletter
\let\@fnsymbol\@arabic
\makeatother

\newcommand{\R}{\mathbb{R}}
\newcommand{\Q}{\mathbb{Q}}

\usepackage[english]{babel}

\newtheorem{theorem}{Theorem}[section]
\newtheorem{lemma}[theorem]{Lemma}

\newtheorem{proposition}[theorem]{Proposition}
\newtheorem{corollary}[theorem]{Corollary}

\theoremstyle{definition}
\newtheorem{examp}[theorem]{Example}
\newtheorem{defn}{Definition}[section]

\newtheorem*{examp*}{Example}
\newtheorem{problem}{Problem}

\newtheorem{remark}{Remark}[section]

\definecolor{donkerblauw}{RGB}{0,0,255}
\definecolor{donkergroen}{RGB}{46,148,0}
\definecolor{donkerrood}{RGB}{204,0,0}
\definecolor{donkergroen2}{RGB}{0,95,0}

\theoremstyle{plain}
\newfloat{Algorithm}{!hbt}{alg}
\newcommand{\T}{^{\sf T}}

\newcounter{thm}[section]

\let\OLDthebibliography\thebibliography
\renewcommand\thebibliography[1]{
  \OLDthebibliography{#1}
  \setlength{\parskip}{0pt}
  \setlength{\itemsep}{0pt plus 0.3ex}
}

\title{\Large Sum-of-squares certificates for symmetric polynomials on the hypercube: {\large A counterexample to a conjecture of De Klerk and Laurent}}

\author{Sven Polak\thanks{Tilburg University, The Netherlands. E-mail: \href{mailto:s.c.polak@tilburguniversity.edu}{\texttt{s.c.polak@tilburguniversity.edu}}}}
\date{June 1, 2026}
\begin{document}

\maketitle

\begin{abstract}This paper studies sum-of-squares (SoS) representations of nonnegative polynomials over the hypercube~$[0,1]^n$. De Klerk and Laurent (SIAM J. Optim., 2010) conjectured that the smallest constant~$C_n$ such that the polynomial~$x_1\cdots x_n +C_n$ is contained in the degree-$n$ truncated quadratic module~$M_{n,n}(x_1-x_1^2,\ldots,x_n-x_n^2)$ of the hypercube is~$C_n=1/(n(n+2))$, for~$n$ even. We specialize symmetry reduction techniques for finding sum-of-squares certificates to the hypercube, where the generators~$x_i-x_i^2$ are not individually invariant under the symmetric group but form an invariant set, and apply them to this conjecture.
Combining this reduction with a further (heuristic) sparsity reduction, a rational rounding step, and an exact verification over~$\mathbb{Q}$, we prove the bound $C_8\leq 11/1000 <1/80$. In particular, this disproves the conjectured optimal value for $n=8$.
\end{abstract}

\section{Introduction}

Consider the hypercube $H_n=[0,1]^n$ described by the inequalities $g_i=x_i-x_i^2 \geq 0$ for $i=1,\ldots,n$. For even $r=2d$, let $M_{n,r}(\mathbf g)$ be the degree-$r$ truncated quadratic module generated by $g_1,\ldots,g_n$, that is, the set of all polynomials of the form
$$ 
p=\sigma_0+\sum_{i=1}^n \sigma_i (x_i-x_i^2),
$$
where $\sigma_0,\sigma_1,\ldots,\sigma_n$ are sums of squares of polynomials, $\deg\sigma_0\le r$, and $\deg\sigma_i\le r-2$.

De Klerk and Laurent~\cite{KlerkLaurent} conjectured that for even $n$, the smallest constant $C_n$ such that the polynomial
$$
x_1 x_2 \cdots x_n + C_n
$$
has a degree-$n$ sum-of-squares representation in $M_{n,n}(\mathbf{g})$ is precisely $C_n = 1/(n(n+2))$. They showed that such a constant $C_n \leq 1$ exists, and conjectured that the smallest one is $C_n = 1/(n(n+2))$, suggested by their numerical verification for the cases~$n=2,4,6$ and a symbolic expression for the case~$n=2$. The main result of this paper is that the conjectured constant is not optimal already for $n=8$.

\begin{theorem} \label{thm:conjecturecounterexample}
For $n=8$, the minimum constant $C_{8}$ such that $x_1x_2\cdots x_8 + C_{8} \in M_{8,8}(\mathbf g)$ satisfies 
$$
C_{8} \le \tfrac{11}{1000} < \tfrac{1}{8\cdot 10}.
$$
In particular, this disproves the conjectured optimal value of De Klerk and Laurent.
\end{theorem}
We emphasize that Theorem~\ref{thm:conjecturecounterexample} disproves the claimed optimality of the constant $1/(n(n+2))$, not the weaker membership statement $x_1\cdots x_n+\frac1{n(n+2)}\in M_{n,n}(\mathbf g)$ which is the statement relevant for the applications to Putinar-type error bounds. If $x_1\cdots x_n+C\in M_{n,n}(\mathbf g)$, then $x_1\cdots x_n+C'\in M_{n,n}(\mathbf g)$ for every $C'\ge C$, since nonnegative constants are sums of squares. Thus, for $n=8$, our certificate with $C=11/1000<1/80$ immediately implies membership for the conjectured constant $1/80$. 

\paragraph{Motivation.}
The conjecture was proposed as a possible route to explicit Putinar-type error bounds on the hypercube. De Klerk and Laurent showed that constants $C_r$ satisfying
\begin{align}\label{eq:membership}
x_1\cdots x_r+C_r\in M_{r,r}(\mathbf g)
\end{align}
can be used to convert certain Bernstein/preordering certificates into certificates in the quadratic module. In particular, for Bernstein order~$k$ and even $r\ge nk$, their argument gives
$$
B_k(p-p_{\min,H_n})+(p_{\max,H_n}-p_{\min,H_n})C_r2^{nk}\in M_{n,r}(\mathbf g),
$$
where $B_k$ denotes the degree-$k$ Bernstein approximation on~$[0,1]^n$. For the applications to Putinar-type error bounds, the essential point is the upper-bound/membership statement~\eqref{eq:membership} with $C_r=1/(r(r+2))$, which would provide a monomial correction constant of order $O(r^{-2})$. The stronger claim that $1/(r(r+2))$ is the minimal such constant is the part disproved in this paper.

This connection was further developed by Magron~\cite{Magron}, who used the relation between the preordering and the quadratic module on the hypercube to derive explicit Putinar-type error bounds, building on the Bernstein approach of De Klerk and Laurent. Recent work on quantitative Putinar/Lasserre bounds on the hypercube includes papers by Baldi and Slot~\cite{BaldiSlot} and by Gribling, de Klerk and Vera~\cite{GriblingKlerkVera}. These papers study general degree bounds and convergence rates for polynomial optimization on the hypercube, while the question studied here is the value of the specific De Klerk--Laurent correction constant.

\paragraph{Contribution.}
We prove Theorem~\ref{thm:conjecturecounterexample} by constructing an exact rational sum-of-squares certificate. The symmetry reduction used here is an application of the standard representation-theoretic reduction for invariant sums of squares, as in Gatermann and Parrilo~\cite{gatpar}. For symmetric polynomial optimization and symmetric sums of squares, closely related $S_n$-reductions are described in~\cite{riener,blekherman}. What is specific to the present paper is the adaptation of these ideas to the quadratic module of the hypercube, where the generators $g_i=x_i-x_i^2$ are not individually invariant but form an $S_n$-invariant set.

The construction has three steps. First, we show in Proposition~\ref{prop:invprop} that a symmetric certificate may be chosen as one $S_n$-invariant sum of squares plus translates of one $S_{n-1}$-invariant multiplier. Second, we block diagonalize the corresponding Gram matrices using representation theory of the symmetric group. Third, we compute numerical SDP values for $C_n$ up to $n=12$, and use a numerical sparsification step only to identify a small support for the $n=8$ certificate. The final certificate is then rationalized and verified exactly over $\mathbb Q$.

 \paragraph{Structure.}  The paper is organized as follows. In Section~\ref{sec:symmetry} we show that symmetric certificates can be written as one $S_n$-invariant sum of squares plus a sum involving translates of one $S_{n-1}$-invariant sum of squares. In Section~\ref{sec:blockdiag} we recall the block diagonalization used to obtain the reduced SDP. Section~\ref{sec:numresults} describes the numerical computations and the sparse support used for the $n=8$ certificate. Section~\ref{sec:rational} explains the exact rational verification. Finally, Section~\ref{sec:ideal} discusses a related Boolean quotient/preordering relaxation suggested by the computations.

\section{Symmetry reduction on the hypercube}\label{sec:symmetry}

Suppose that~$p\in \mathbb{R}[x_1,\ldots,x_n]$ is a polynomial symmetric under the action of~$S_n$ on the indices of the variables: $\pi\cdot x_i=x_{\pi(i)}$, and that we aim to obtain a degree~$2d$ sum-of-squares certificate on the hypercube:
\begin{align}\label{toprove}
p = \sigma_0 + \sum_{i=1}^n \sigma_i (x_i-x_i^2), \quad \text{with $\deg \sigma_0 \leq 2d$, and~$\deg \sigma_i \leq 2d-2$ for~$i=1,\ldots,n$.}
\end{align}
The polynomials~$x_i-x_i^2$ are themselves not invariant under~$S_n$, but the set~$\mathbf{g}:= \{x_1-x_1^2,\ldots,x_n-x_n^2\}$ is invariant. For~$j \in [n]$, let~$\mathrm{Stab}_{S_n}(j)$ be the stabilizer subgroup of~$j$ in~$S_n$, so $\mathrm{Stab}_{S_n}(j)= \{\pi \in S_n : \pi(j) = j \}$. Note that this group is isomorphic to~$S_{n-1}$.

The following proposition specializes standard symmetry-reduction arguments (cf., e.g., Gatermann \& Parrilo~\cite{gatpar}) to the hypercube.

\begin{proposition}\label{prop:invprop}
Let $p \in \mathbb{R}[x_1,\ldots,x_n]$ be invariant under the symmetric group $S_n$. Then $p$ is contained in the degree~$2d$ truncated quadratic module~$M_{n,2d}(\mathbf{g})$ over the hypercube if and only if there exist sum-of-squares polynomials $ \widetilde{\sigma}$, $ \widetilde{\rho}$ with $\deg  \widetilde{\sigma} \leq 2d$, and~$\deg \widetilde{\rho}\leq 2d-2$, such that
\begin{align}\label{eq:decomphypercube}
p = \widetilde{\sigma} + \sum_{i=1}^n ((i\,n) \cdot \widetilde{\rho})  (x_i - x_i^2),
\end{align}
where $ \widetilde{\sigma}$ is $S_n$-invariant and $ \widetilde{\rho}$ is invariant under~$\mathrm{Stab}_{S_n}(n)$.
\end{proposition}
\begin{proof}
One direction is clear. If $p$ admits a decomposition as in~\eqref{eq:decomphypercube} where $\widetilde{\sigma}$ and $\widetilde{\rho}$ are sums of squares with the stated degree bounds, then each $(i\,n)\cdot\widetilde{\rho}$ is again a sum of squares of degree at most $2d-2$. Hence $p\in M_{n,2d}(\mathbf g)$.

Conversely, suppose $p$ admits a sum-of-squares certificate of the form $p = \sigma_0 + \sum_{i=1}^n \sigma_i (x_i - x_i^2)$, where each $\sigma_i$ is a sum-of-squares polynomial with $\deg \sigma_0 \leq 2d$ and $\deg \sigma_i \leq 2d - 2$. Since $p$ is $S_n$-invariant, averaging over~$S_n$ gives
\begin{align*}
p &= \frac{1}{n!} \sum_{\pi \in S_n} \left( \pi \cdot \sigma_0 + \sum_{i=1}^n (\pi \cdot \sigma_i)(x_{\pi(i)} - x_{\pi(i)}^2) \right)
\\&= \frac{1}{n!} \sum_{\pi \in S_n} \left( \pi \cdot \sigma_0 \right)+ \sum_{i=1}^n\frac{1}{n!} \sum_{\pi \in S_n} \left( (\pi \cdot \sigma_i)(x_{\pi(i)} - x_{\pi(i)}^2) \right).
\end{align*}
Define $\widetilde{\sigma} := \frac{1}{n!} \sum_{\pi \in S_n} \pi \cdot \sigma_0$, which is $S_n$-invariant and a sum-of-squares polynomial with degree at most $2d$. Furthermore, for~$i=1,\ldots,n$ we define $\widetilde{\sigma}_i := \frac{1}{(n-1)!} \sum_{\pi \in \mathrm{Stab}_{S_n}(i)} \pi \cdot \sigma_i$.
\begin{align*}
p &= \widetilde{\sigma}+  \sum_{i=1}^n\frac{1}{n!} \sum_{\pi \in S_n} \left( (\pi \cdot \sigma_i)(x_{\pi(i)} - x_{\pi(i)}^2) \right)
\\&= \widetilde{\sigma}+  \sum_{i=1}^n \frac{1}{n} \sum_{j=1}^n(i \; j) \cdot \left(\frac{1}{(n-1)!}\sum_{\tau \in \mathrm{Stab}_{S_n}(i)} (\tau \cdot \sigma_i) (x_i-x_i^2)\right)
\\&= \widetilde{\sigma}+  \sum_{i=1}^n \frac{1}{n} \sum_{j=1}^n(i \; j) \cdot\left(  \widetilde{\sigma}_i (x_i-x_i^2)\right)
= \widetilde{\sigma}+  \sum_{i=1}^n \frac{1}{n} \sum_{j=1}^n \left((i j)  \widetilde{\sigma}_i\right) (x_j-x_j^2)
\\&= \widetilde{\sigma}+  \sum_{j=1}^n \left(\sum_{i=1}^n \frac{1}{n} (i j)  \widetilde{\sigma}_i\right) (x_j-x_j^2).
\end{align*}
Now we define for~$j=1,\ldots,n$, the polynomial $\rho_j:=\left(\sum_{i=1}^n \frac{1}{n} (i \; j)\cdot  \widetilde{\sigma}_i\right) $. Each $\rho_j$ is a sum of squares of degree $\leq 2d-2$, since sums of squares are preserved under permutations of the variables and under nonnegative linear combinations. Then each~$\rho_j$ is invariant under~$\mathrm{Stab}_{S_n}(j)$, as  the separate terms $(i \; j)\cdot  \widetilde{\sigma}_i$ are invariant under this group. Moreover, we have~$(j \;n) \rho_n = \rho_j$ for every~$j=1,\ldots,n$. To see this, write $(j\,n)\cdot\rho_n = \sum_{i=1}^n\frac1n (j\,n)(i\,n)\cdot\widetilde{\sigma}_i$. We compare this term by term with $\rho_j=\sum_{i=1}^n\frac1n (i\,j)\cdot\widetilde{\sigma}_i$. For fixed $i$, the permutation $(i\,j)(j\,n)(i\,n)$ fixes $i$, and hence belongs to $\mathrm{Stab}_{S_n}(i)$. Therefore, using that $\widetilde{\sigma}_i$ is $\mathrm{Stab}_{S_n}(i)$-invariant, we have
$$
(j\,n)(i\,n)\cdot\widetilde{\sigma}_i =(i\,j)\cdot\left((i\,j)(j\,n)(i\,n)\cdot\widetilde{\sigma}_i\right) = (i\,j)\cdot\widetilde{\sigma}_i.
$$
So indeed $(j \;n) \rho_n = \rho_j$ for every~$j=1,\ldots,n$. Hence, by setting~$\widetilde{\rho}:=\rho_n$, we conclude that
\begin{align*}
p &= \widetilde{\sigma} + \sum_{j=1}^n \rho_j  (x_j - x_j^2)= \widetilde{\sigma} + \sum_{j=1}^n ((j\;n) \cdot {\widetilde{\rho}})  (x_j - x_j^2),
\end{align*}
which gives the desired expression.
\end{proof}

Let $[x]_d$ denote the column vector of all monomials in $x_1,\ldots,x_n$ of total degree at most~$d$. For any matrix~$M$ indexed by monomials~$u$,$v$, we write $(\pi M)_{u,v} = (M)_{\pi u, \pi v}$ for all $\pi \in S_n$. 

\begin{corollary}
Let $p \in \mathbb{R}[x_1,\ldots,x_n]$ be invariant under the symmetric group $S_n$. Then $p$ is contained in the degree~$2d$ truncated quadratic module~$M_{n,2d}(\mathbf{g})$ over the hypercube if and only if there exists an~$S_n$-invariant matrix~$Q \in \R^{[x]_d \times [x]_d}$ and a~$\mathrm{Stab}_{S_n}(n)$-invariant matrix~$R \in \R^{[x]_{d-1} \times [x]_{d-1}}$ such that
\begin{align*}
&Q \succeq 0, \quad R\succeq 0, \quad\quad\text{and}\\
&p = [x]_d\T Q [x]_d + \sum_{i=1}^n \left((i \; n) \cdot [x]_{d-1}\T  R [x]_{d-1} \right) (x_i-x_i^2).
\end{align*}
\end{corollary}
\begin{proof}
This follows directly from Proposition~\ref{prop:invprop}, by rewriting the sum-of-squares polynomials~$\widetilde{\sigma}= [x]_d\T Q [x]_d $ and~$\widetilde{\rho}=  [x]_{d-1}\T  R [x]_{d-1}$, with $Q \succeq 0$ and $R \succeq 0$. Since $\widetilde{\sigma}$ and $\widetilde{\rho}$ are invariant under $S_n$ and $\mathrm{Stab}_{S_n}(n)$, respectively, we may assume that $Q$ and $R$ are invariant under the corresponding groups by replacing them with their group averages.
\end{proof}

\begin{examp}[The conjecture for~$n=4$]
De Klerk and Laurent~\cite{KlerkLaurent} gave an explicit expression for~$n=2$: 
$$
8\, x_1x_2+1 = (1-2x_1-2x_2)^2 + 4 (x_1-x_1^2)+ 4 (x_2-x_2^2).
$$
They furthermore verified the conjecture numerically for~$n=2,4,6$. 
 
We used the above symmetry-reduction method to obtain an explicit symbolic expression for~$24 x_1x_2x_3x_4+1$ as an element of the truncated quadratic module~$M_{4,4}(\mathbf{g})$ over the hypercube. The identity below was obtained by solving the symmetry-reduced SDP for $n=4$ and extracting a human-readable certificate from one feasible Gram matrix. See Section~\ref{subsec:n4-cert} for details. 
The equality can be verified by direct expansion:
$$
 24\, x_1x_2x_3x_4 +1= \sigma_0 + \sum_{i=1}^4 \sigma_i \cdot  (x_i-x_i^2),
$$
where
\begin{align*}
 \sigma_0&= \left(1+ 2\hspace{-5pt} \sum_{1\leq i<j\leq 4}\hspace{-5pt}  x_ix_j -\sum_{i=1}^4 x_i^2 -\sum_{i=1}^4x_i \right)^2 + \left( \sum_{i=1}^4 x_i^2 -\sum_{i=1}^4x_i \right)^2\\ & \phantom{=} \,\,+2 \hspace{-5pt} \sum_{1 \leq i < j \leq 4} \left(x_i^2 -x_j^2 -x_i +x_j\right)^2,
 \\ \sigma_i &= 2 \left(1-x_i-\tfrac{1}{2} \sum_{\substack{j \in [4]:\\ j \neq i}}x_j \right)^2 +6\left(x_i-\tfrac{1}{2} \sum_{\substack{j \in [4]:\\ j\neq i}} x_j\right)^2  \text{for $i \in [4]$.}
\end{align*}
\end{examp}

We next recall the block diagonalization used for the matrices $Q$ and $R$.

\section{Block-diagonalization via representation theory}\label{sec:blockdiag}

We recall the representation-theoretic framework underlying the symmetry reduction used in this paper. This goes back to Gatermann--Parrilo~\cite{gatpar}; see also, e.g.,~\cite{invsemi,vallentin,gijswijt,LPS}. For the symmetric-group reductions of $S_n$-invariant polynomials, see~\cite{blekherman,sagan,riener}.

Let $G$ be a finite group acting on a finite-dimensional complex vector space $V$. We call $V$ a \emph{$G$-module}. A subspace $W\subseteq V$ is a \emph{$G$-submodule} if it is $G$-invariant. A linear map $\psi:V\to W$ between $G$-modules is a \emph{$G$-homomorphism} if $\psi(g\cdot v)=g\cdot \psi(v)$ for all $g\in G$ and $v\in V$.

Since $G$ is finite, $V$ is completely reducible and decomposes as
$$
V=\bigoplus_{i=1}^k \bigoplus_{j=1}^{m_i} V_{i,j},
$$
where each $V_{i,j}$ is irreducible, and for fixed $i$ the submodules $V_{i,1},\ldots,V_{i,m_i}$ are mutually $G$-isomorphic (different $i$ correspond to non-isomorphic irreducible types).

\begin{defn}[Representative set]\label{def:reprset}
For each $i\in[k]$ and $j\in[m_i]$ choose a nonzero vector $u_{i,j}\in V_{i,j}$ such that for every $j,j'\in[m_i]$ there exists a $G$-isomorphism $\varphi_{j\to j'}:V_{i,j}\to V_{i,j'}$ satisfying $\varphi_{j\to j'}(u_{i,j})=u_{i,j'}$. Define $U_i:=(u_{i,1},\ldots,u_{i,m_i})$. A collection $\{U_1,\ldots,U_k\}$ obtained in this way is called a \emph{representative set} for the action of $G$ on $V$.
\end{defn}
In the applications below the $G$-modules are polynomial spaces with real monomial bases with the standard Euclidean inner product, and $G$ is either $S_n$ or $S_{n-1}$. The representative vectors constructed below are real. Thus the reduction can be written entirely over $\mathbb R$, and we do so from now on. For a polynomial $f\in\mathbb R[x_1,\ldots,x_n]$, write
$$
\mathcal S_G(f):=\sum_{g\in G}g\cdot f.
$$
We will use the  block reduction for sum-of-squares polynomials in the following form, due to Blekherman and Riener~\cite[Remark 4.3 and Corollary 4.4]{blekherman}; see also~\cite{gatpar,riener}.
\footnote{Blekherman and Riener use the normalized Reynolds operator $R_G=|G|^{-1}\mathcal S_G$; this only rescales the positive semidefinite
matrices $Q_\lambda$. We do not normalize, so that our SDP constraint matrices consist of integers.}
\begin{theorem}[Block form of invariant sums of squares]
\label{thm:invariant-sos-block}
Let $G$ be either $S_n$, acting on $V=\mathbb R[x_1,\ldots,x_n]_{\le d}$ by permuting all variables, or $S_{n-1}$, acting on $V= \mathbb R[x_1,\ldots,x_n]_{\le d}$ by permuting $x_1,\ldots,x_{n-1}$ and fixing $x_n$. Let $U_\lambda=(u_{\lambda,1},\ldots,u_{\lambda,m_\lambda})$ be a representative set for the $G$-module $V$. Define
$$
X_\lambda(x):=\left(\mathcal S_G(u_{\lambda,a}u_{\lambda,b})\right)_{a,b=1}^{m_\lambda}.
$$
Then a $G$-invariant polynomial $\sigma$ is a sum of squares of polynomials
from $V$ if and only if there exist positive semidefinite matrices
$Q_\lambda\succeq 0$ such that
$$
    \sigma(x)=\sum_\lambda \langle Q_\lambda,X_\lambda(x)\rangle .
$$
\end{theorem}

\subsection{Representative vectors for \texorpdfstring{$S_n$}{Sn} acting on \texorpdfstring{$\mathbb{R}[x_1,\ldots,x_n]_{\le d}$}{R[x1,...xn]leqd}}\label{subsec:repr-sn}

We now describe an explicit combinatorial construction of a representative set for the action of $S_n$ on
$$
V_d:=\mathbb{R}[x_1,\ldots,x_n]_{\le d}.
$$
We work in the standard monomial basis and identify $V_d\cong\mathbb{R}^{Z_d}$ by coefficient vectors, where $Z_d$ is the set of all monomials of total degree at most $d$.

Let $x^\alpha=x_1^{\alpha_1}\cdots x_n^{\alpha_n}$ be a monomial with $\sum_i \alpha_i\le d$. Let
$$
a_0<a_1<\cdots<a_s
$$
be the distinct exponent values appearing among $\alpha_1,\ldots,\alpha_n$, and set
$$
\eta_\ell:=|\{i\in[n]:\alpha_i=a_\ell\}|,\qquad \ell=0,\ldots,s.
$$
Then $\eta=(\eta_0,\ldots,\eta_s)$ is a composition of $n$. We call
$$
\xi=(a_0<\cdots<a_s;\eta), \qquad \eta=(\eta_0,\ldots,\eta_s),
$$
the exponent pattern of $x^\alpha$.

Two monomials lie in the same $S_n$-orbit if and only if they have the same exponent pattern. Let $\Xi_{n,d}$ be the set of exponent patterns arising from monomials of degree at most $d$; equivalently, $\Xi_{n,d}$ consists of all $\xi=(a_0<\cdots<a_s;\eta)$ for some~$s\geq 0$, $\eta=(\eta_0,\ldots,\eta_s)$, where $a_0,\ldots,a_s$ are nonnegative integers and $\eta_0,\ldots,\eta_s$ are positive integers satisfying $\sum_{\ell=0}^s\eta_\ell=n$, and $\sum_{\ell=0}^s \eta_\ell a_\ell\le d$. The set $\Xi_{n,d}$ corresponds bijectively to the $S_n$-orbits of monomials in $V_d$. 

Fix a partition $\lambda\vdash n$ and an exponent pattern $\xi=(a_0<\cdots<a_s;\eta)\in\Xi_{n,d}$. A \emph{semistandard Young tableau of shape $\lambda$ and content $\eta$} is a filling of the Young diagram of $\lambda$ with symbols from $\{0,1,\ldots,s\}$ such that symbol $\ell$ appears exactly $\eta_\ell$ times, each row is weakly increasing (from left to right), and each column is strictly increasing (from top to bottom).

Let $t_\lambda$ be the Young tableau of shape $\lambda$ filled row-by-row with $1,2,\ldots,n$. Let $C_{t_\lambda}$ be the subgroup of $S_n$ consisting of all permutations of the entries within the columns of $t_\lambda$. For each $c\in C_{t_\lambda}$, let $t_c=c\cdot t_\lambda$ be the numbered tableau obtained by applying $c$ to the entries of $t_\lambda$.

\paragraph{From \texorpdfstring{$(\lambda,T,\xi)$}{(lambda,T,xi)} to a polynomial.}

Fix $\xi=(a_0<\cdots<a_s;\eta)\in\Xi_{n,d}$, and a semistandard tableau $T$ of shape $\lambda$ and content $\eta$. The symbols $\{0,1,\ldots,s\}$ appearing in $T$ are identified with the exponent levels $a_0,\ldots,a_s$ in this fixed order.\footnote{The increasing order $a_0<\cdots<a_s$ makes the exponent pattern canonical. For the construction itself, what matters is a fixed identification between tableau symbols and exponent levels; a simultaneous relabelling of the symbols and exponent levels gives an equivalent representative set.}

Given a filling $S$ of shape $\lambda$ with symbols in $\{0,\ldots,s\}$ and a numbered tableau $t$ of shape $\lambda$, define
$$
    m_{S,\xi}(t):=\prod_{b\in\lambda} x_{t(b)}^{a_{S(b)}} ,
$$
where $b$ runs over the boxes of the Young diagram, $S(b)$ is the symbol in box $b$, and $t(b)$ is the number in box $b$.

Finally, we sum over all distinct fillings $T'$ obtained by permuting the entries within each row of $T$:
$$
u_{\lambda,T,\xi} := \sum_{T'\sim T}\sum_{c\in C_{t_\lambda}} \operatorname{sgn}(c)\,m_{T',\xi}(c\cdot t_\lambda).
$$
The polynomial $u_{\lambda,T,\xi}$ is a signed linear combination of monomials all having the same orbit data $\xi=(a_0<\cdots<a_s;\eta)$.

Collecting the vectors $u_{\lambda,T,\xi}$ over all $\xi=(a_0<\cdots<a_s;\eta)\in\Xi_{n,d}$ and over all semistandard tableaux $T$ of shape $\lambda$ and content $\eta$, yields the matrix $U_\lambda$. For fixed $\lambda$, the vectors $u_{\lambda,T,\xi}$ corresponding to different $(T,\xi)$ generate mutually $S_n$-isomorphic copies of the Specht module $S^\lambda$. The matrices  $\{U_\lambda\}_{\lambda\vdash n}$ together form a representative set for the action of $S_n$ on $\mathbb{R}[x_1,\ldots,x_n]_{\leq d}$ cf.\ Definition~\ref{def:reprset}, see e.g.,~\cite{sagan, riener}.

\begin{examp}
We illustrate the construction by reproducing the representative polynomials used in the block $Q_{(6,2)}$ of Table~\ref{table:certificate} below. Let $n=8$ and $\lambda=(6,2)$, with canonical tableau
$$
t_\lambda={\begin{ytableau}
1 & 2 & 3 & 4 & 5 & 6\\
7 & 8
\end{ytableau}}.
$$
The only columns of height two are the first two, so $C_{t_\lambda}=\langle (1\,7),(2\,8)\rangle$ has four elements with signs $+,-,-,+$. Consider the exponent levels $a_0=0,a_1=1$. For content $\eta^{(1)}=(6,2)$ there is a unique semistandard tableau
$$
T_1={\begin{ytableau}
0&0&0&0&0&0\\
1&1
\end{ytableau}}.
$$
With $\xi_1=(0<1;\eta^{(1)})=(0<1;(6,2))$, the construction gives 
$$
u_{(6,2),T_1,\xi_1}=x_7x_8-x_2x_7-x_1x_8+x_1x_2.
$$
For content $\eta^{(2)}=(5,3)$ one takes
$$
T_2={\begin{ytableau}
0&0&0&0&0&1\\
1&1
\end{ytableau}}.
$$
With $\xi_2=(0<1;\eta^{(2)})=(0<1;(5,3))$, the row-equivalence sum yields
$$
u_{(6,2),T_2,\xi_2} = (x_7x_8-x_2x_7-x_1x_8+x_1x_2)\sum_{i=3}^6 x_i.
$$
For content $\eta^{(3)}=(4,4)$ one takes
$$
T_3={\begin{ytableau}
0&0&0&0&1&1\\
1&1
\end{ytableau}}.
$$
With $\xi_3=(0<1;\eta^{(3)})=(0<1;(4,4))$, one obtains
$$
u_{(6,2),T_3,\xi_3} = (x_7x_8-x_2x_7-x_1x_8+x_1x_2)\sum_{3\le i<j\le 6} x_i x_j.
$$
These are exactly the three polynomials used for the block $Q_{(6,2)}$ in Table~\ref{table:certificate}. Further vectors in the $(6,2)$-block of $V_4$ arise from non-squarefree exponent patterns. They are produced by the same tableau construction and are present in the full symmetry-reduced SDP, but are not needed in our sparse certificate.
\end{examp}

\subsubsection{Representative vectors for \texorpdfstring{$S_{n-1}$}{Sn-1} acting on \texorpdfstring{$\mathbb{R}[x_1,\ldots,x_n]_{\leq d-1}$}{R[x1,...,xn]leq d-1}}
To reduce the multiplier block~$R$ matrix, we consider the group $S_{n-1}$, acting on $x_1,\ldots,x_{n-1}$ and fixing $x_n$. Thus 
$$
\mathbb R[x_1,\ldots,x_n]_{\le d-1} = \bigoplus_{q=0}^{d-1} x_n^q\,\mathbb R[x_1,\ldots,x_{n-1}]_{\le d-1-q}
$$
as an $S_{n-1}$-module. For each partition $\mu\vdash n-1$, the representative set $V_\mu$ is obtained by concatenating, over all $q=0,\ldots,d-1$, the representative vectors of type $\mu$ in $\mathbb R[x_1,\ldots,x_{n-1}]_{\le d-1-q}$, each multiplied by $x_n^q$.

\subsection{The symmetry-reduced SDP}

\begin{proposition}[Block-reduced hypercube SDP] \label{thm:block-reduced-sdp}
Let $p\in \mathbb R[x_1,\ldots,x_n]_{\le 2d}$ be $S_n$-invariant. Choose representative sets $U_\lambda=(u_{\lambda,1},\ldots,u_{\lambda,m_\lambda})$, $\lambda\vdash n$, for the action of $S_n$ on $\mathbb R[x_1,\ldots,x_n]_{\le d}$. Furthermore, choose representative sets $V_\mu=(v_{\mu,1},\ldots,v_{\mu,r_\mu})$, $\mu\vdash n-1$, for the action of $S_{n-1}$ on $\mathbb R[x_1,\ldots,x_n]_{\le d-1}$ fixing $x_n$. Define
$$
X_\lambda(x) :=\left(\mathcal S_{S_n}(u_{\lambda,a}u_{\lambda,b})\right)_{a,b=1}^{m_\lambda}, \quad \text{and} \quad  Y_\mu(x) :=\left(\mathcal S_{S_{n-1}}(v_{\mu,a}v_{\mu,b})\right)_{a,b=1}^{r_\mu}.
$$
Then $p\in M_{n,2d}(\mathbf g)$, where $g_i=x_i-x_i^2$, if and only if there exist positive semidefinite matrices
$$
Q_\lambda \succeq 0 \quad (\lambda\vdash n), \qquad R_\mu\succeq0\quad(\mu\vdash n-1),
$$
such that
\begin{align} \label{eq:coefficientmatchformula} p =\sum_{\lambda\vdash n} \left\langle Q_\lambda, X_\lambda(x)\right\rangle+
\sum_{i=1}^n\left( (i\,n)\cdot\left(\sum_{\mu\vdash n-1}\left\langle R_\mu,  Y_\mu(x)\right\rangle\right)\right)(x_i-x_i^2).
\end{align}
\end{proposition}

\begin{proof}
By Proposition~\ref{prop:invprop}, every symmetric degree-$2d$
certificate may be written as $p=\sigma+\sum_{i=1}^n ((i\,n)\rho)(x_i-x_i^2)$, where $\sigma$ is an $S_n$-invariant sum of squares from $\mathbb R[x_1,\ldots,x_n]_{\le d}$, and $\rho$ is an $S_{n-1}$-invariant sum of squares of polynomials of degree at most $d-1$, with $S_{n-1}$ fixing $x_n$. Apply Theorem~\ref{thm:invariant-sos-block} to $\mathbb R[x_1,\ldots,x_n]_{\le d}$ with $G=S_n$. This gives
$$
\sigma=\sum_{\lambda\vdash n}\left\langle Q_\lambda, X_\lambda(x)\right\rangle,\qquad Q_\lambda\succeq0.
$$
Apply the same theorem to $\mathbb R[x_1,\ldots,x_n]_{\le d-1}$ with $G=S_{n-1}$, where $S_{n-1}$ acts
on $x_1,\ldots,x_{n-1}$ and fixes $x_n$. This gives
$$
\rho =\sum_{\mu\vdash n-1}\left\langle R_\mu, Y_\mu(x)\right\rangle, \qquad R_\mu\succeq0.
$$
Substituting these expressions for $\sigma$ and $\rho$ gives the stated block-reduced form. The converse is immediate: positive semidefinite blocks give sums of squares, and Proposition~\ref{prop:invprop} then gives a valid certificate in $M_{n,2d}(\mathbf g)$.
\end{proof}

To obtain the SDP, we expand Equation~\eqref{eq:coefficientmatchformula} in Proposition~\ref{thm:block-reduced-sdp} and equate coefficients. Since both sides are $S_n$-invariant, it suffices to impose one equation for each orbit of monomials under $S_n$. In the implementation we use the sorted exponent vector as the orbit representative.

\begin{remark}
The exact verification of the certificate only uses the easy direction of Proposition~\ref{thm:block-reduced-sdp}. Once rational matrices $Q_{\lambda}$ and~$R_{\mu}$ are given, it suffices to check that they are positive semidefinite and that \eqref{eq:coefficientmatchformula} holds coefficientwise. Positive semidefinite blocks give sums of squares directly. So the exact verification does not depend on Theorem~\ref{thm:invariant-sos-block}; that theorem is used to obtain the reduced SDP and find the certificate.
\end{remark}

 \subsection{A small example: extracting the \texorpdfstring{$n=4$}{n = 4} certificate from the SDP}
\label{subsec:n4-cert}

We explain how the preceding human-readable certificate for $n=4$ was obtained from the block-reduced SDP. This illustrates how the SDP output can be converted into an explicit sum-of-squares identity.

We solved the block-reduced SDP for $24x_1x_2x_3x_4+1\in M_{4,4}(\mathbf g)$. First, we restricted the multiplier block $R$ to its trivial $S_3$-component by imposing $R_{(2,1)}=0$. Equivalently, the polynomial $\rho$ is represented using only squares of $S_3$-invariant polynomials, where $S_3$ acts on $x_1,x_2,x_3$ and fixes $x_4$. Second, in the trivial $S_4$-block $Q_{(4)}$, the diagonal of the numerical solution, after multiplying by $24$, was close to $1,\ 2,\ 2,\ 4$. This suggested fixing the two middle diagonal entries by adding the constraints $(Q_{(4)})_{2,2}=(Q_{(4)})_{3,3}=\frac2{24}$. With these three additional linear constraints, the SDP rationalizes to the simple blocks described below.

For the trivial $S_4$-block, the representative set is $$U_{(4)}=\left(1,\ \sum_{i=1}^4 x_i,\ \sum_{i=1}^4 x_i^2,\ \sum_{1\le i<j\le 4}x_ix_j \right).$$ The recovered $Q_{(4)}$-block is $$Q_{(4)} =\frac1{24} \begin{pmatrix}
1&-1&-1&2\\
-1&2&0&-2\\
-1&0&2&-2\\
2&-2&-2&4
\end{pmatrix}.
$$
Since the entries of $U_{(4)}$ are already $S_4$-invariant, the unnormalized orbit-sum contributes a factor $24$. Thus this block contributes $U_{(4)}^{\mathsf T} 24 Q_{(4)} U_{(4)}$,  which factors using a Cholesky decomposition as
$$
\left(1-\sum_{i=1}^4 x_i-\sum_{i=1}^4x_i^2 +2\sum_{1\le i<j\le4}x_ix_j \right)^2
+
\left(\sum_{i=1}^4x_i^2-\sum_{i=1}^4x_i \right)^2.
$$
The nontrivial $S_4$-block used is the $(3,1)$-block. With basis $\{u_1=x_4-x_1, \, u_2=x_4^2-x_1^2, u_3=x_3x_4+x_2x_4-x_1x_3-x_1x_2 \}$, the recovered block is
$$
Q_{(3,1)} = \frac12
\begin{pmatrix}
1&-1&0\\
-1&1&0\\
0&0&0
\end{pmatrix},
\quad
Q_{(2,2)}=0.
$$
Hence only $u_1-u_2$ is used. This contribution is
$$
\frac12\mathcal S_{S_4}\bigl((u_1-u_2)^2\bigr) = 2\sum_{1\le i<j\le4} \left(x_i^2-x_j^2-x_i+x_j\right)^2.
$$
It remains to describe the polynomial $\rho$. Let $s=x_1+x_2+x_3$, and $U_{(3)}=(1,s,x_4)$. 
The recovered trivial $S_3$-block is
$$
R_{(3)} =\begin{pmatrix} \frac13&-\frac16&-\frac13\\
-\frac16&\frac13&-\frac13\\
-\frac13&-\frac13&\frac43
\end{pmatrix}.
$$
Again, the basis elements are already $S_3$-invariant, so the unnormalized orbit-sum contributes a factor $|S_3|=6$. Thus the actual multiplier is $\rho = U_{(3)}^{\mathsf T} 6 R_{(3)}U_{(3)}$.
This factors as 
$$ \rho = 2\left(1-x_4-\frac12(x_1+x_2+x_3)\right)^2 + 6\left(x_4-\frac12(x_1+x_2+x_3)\right)^2.
$$
Substituting the $Q_{(4)}$- and $Q_{(3,1)}$-contributions into $\sigma_0$, the certificate becomes
$$ 24x_1x_2x_3x_4+1 = \sigma_0 +\sum_{i=1}^4 ((i\,4)\rho)(x_i-x_i^2),$$
which is exactly the certificate displayed above. Thus the readable  identity is obtained from the SDP by restricting $\rho$ to squares of $S_3$-invariant polynomials, imposing two rational diagonal constraints suggested by the numerical solution, and then factoring the resulting low-rank rational blocks using a Cholesky decomposition. The decomposition of $24x_1x_2x_3x_4+1$ is not unique. For instance, if one restricts to the trivial $R_{(3)}$- and $Q_{(4)}$-blocks, another rational certificate can be obtained after fixing a single diagonal entry. We do not display this alternative, because the certificate above has a more transparent form.

\section{Numerical results}\label{sec:numresults}

Let $C_{n,2d}$ be the minimum constant such that $x_1 \ldots x_n + C_{n,2d}$ is contained in the truncated quadratic module $M_{n,2d}(\mathbf{g})$. So $C_{n,n} = C_n$. The following numerical values for $C_{n,2d}$ are found, using the high precision solver \texttt{SDPA-GMP}~\cite{nakata}:

\begin{table}[H]\small
\centering
    \begin{tabular}{| r | r|| p{2.6cm} |p{2.6cm}|}
    \hline
   $n$ & $2d$ & numerical SDP value of $C_{n,2d}$ & suggested value from numerics\\\hline 
       2 & 2   &   0.125000000   & $1/8$\\          
       4 & 4   &   0.041666667 & $1/24 $ \\   
       6 & 6   &   0.020833333 & $1/48$ \\           
       8 & 8   &   0.010717341 &  \\   
       10 & 10 &   0.005566474 & \\ 
       12 & 12 &   0.002817734 & \\\hline 
       2 & 4   &   0.002809325 & \\          
       2&  6   &   0.000253534 & \\   
       2 & 8   &   0.000049163 & \\  
       2 & 10  &   0.000017179 &\\ 
       4 & 6   &   0.004464286 & 1/224\\
       4 & 8   &   0.000535555 & \\
       4 & 10 &    0.000111167 & \\
       6 & 8   &   0.002777778 & 1/360\\ 
       6 & 10  &   0.000310382 & \\   
       8 & 10 &    0.001655191 & \\   
    \hline 
    \end{tabular}
    \caption{\small Numerical SDP values for $C_{n,2d}$, rounded to nine decimal places. Here~$C_{n,2d}$ is the smallest constant such that $x_1 \ldots x_n + C_{n,2d}$ is contained in the truncated quadratic module $M_{n,2d}(\mathbf{g})$.} \label{table:SDPresults}
\end{table}
The values in Table~\ref{table:SDPresults} are numerical SDP optima and are not used as proof of the main theorem.

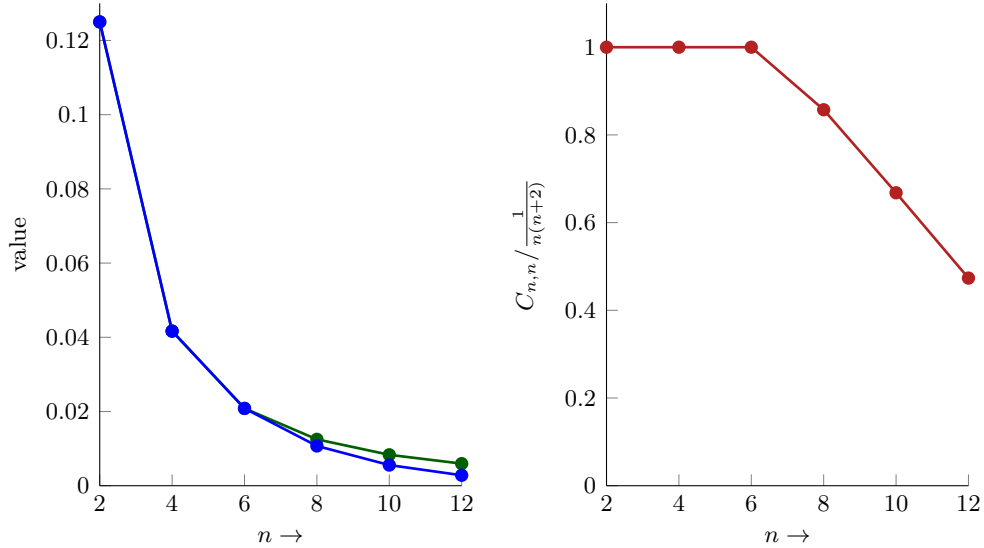
\begin{figure}[H]
\centering
\pgfplotsset{xmin=2, xmax=12, ymin=0, ymax=0.13}   
\begin{tikzpicture}
\begin{axis}[domain=2:12, xtick={2,4,6,8,10,12}, restrict y to domain=0:1,
width={0.4\textwidth}, height={0.5\textwidth},
axis x line*=bottom,axis y line*=left,
    /pgf/number format/.cd,tick label style={font=\footnotesize},
    ylabel = {\footnotesize value},
ylabel style={yshift=0cm,xshift=0cm},
xlabel={$n \rightarrow$},
xlabel style={yshift=0cm,xshift=0cm,font=\footnotesize},   
/pgfplots/scaled ticks=false,
      1000 sep={},
precision=6,
ylabel style={yshift=0cm},
 every axis plot/.append style={thick}]
  \addplot+ [mark=*,    mark options={
        fill=donkergroen2,
        draw=donkergroen2
    },donkergroen2, line width={1}] table {
2 0.125000000
4 0.041666667
6 0.020833333
8 0.012500000
10 0.008333333
12 0.005952381
};
  \addplot+ [mark=*,donkerblauw,     mark options={
        fill=donkerblauw,
        draw=donkerblauw
    }, line width={1}] table {
2 0.125000000
4 0.041666667
6 0.020833333
8 0.010717341
10 0.005566474
12 0.002817734
};
\end{axis}
\end{tikzpicture}
\pgfplotsset{xmin=2, xmax=12, ymin=0, ymax=1.1}   
\begin{tikzpicture}
\begin{axis}[domain=2:12, xtick={2,4,6,8,10,12}, restrict y to domain=0:1,
width={0.4\textwidth}, height={0.5\textwidth},
axis x line*=bottom,axis y line*=left,
    /pgf/number format/.cd,tick label style={font=\footnotesize},
    ylabel = {\footnotesize $C_{n,n}/\tfrac{1}{n(n+2)}$},
ylabel style={yshift=0cm,xshift=0cm},
xlabel={$n \rightarrow$},
xlabel style={yshift=0cm,xshift=0cm,font=\footnotesize},   
/pgfplots/scaled ticks=false,
      1000 sep={},
precision=6,
ylabel style={yshift=0cm},
 every axis plot/.append style={thick}]
  \addplot+ [mark=*,    mark options={
        fill=firebrick,
        draw=firebrick,
    },firebrick, line width={1}] table {
2 1.000
4 1.000
6 1.000
8 0.8574
10 0.6680
12 0.4735
};
\end{axis}
\end{tikzpicture}
\caption{Comparison between the conjectured values $1/(n(n+2))$ (green) and the numerical SDP values obtained for $C_{n,n}$ (blue). The fraction $C_{n,n}/\tfrac{1}{n(n+2)}$ (red) is shown in the plot on the right.}
\end{figure}

\begin{remark}
To compute the larger numerical instances in Table~\ref{table:SDPresults}, we
used the following column-stabilizer simplification. For a representative vector
$u_a=u_{\lambda,T,\xi}$, define the row-symmetrized polynomial~$ f_a:=\sum_{T'\sim T} m_{T',\xi}(t_\lambda)$, so that $u_a=\sum_{c\in C_{t_\lambda}}\operatorname{sgn}(c)\,c\cdot f_a$. Hence, for two representatives $u_a,u_b$ in the same $\lambda$-block,
$$
\mathcal S_{S_n}(u_a u_b)=|C_{t_\lambda}|\,\mathcal S_{S_n}\left( f_a\sum_{e\in C_{t_\lambda}}\operatorname{sgn}(e)\,e\cdot f_b\right).
$$
To see this, after expanding the left-hand side as a sum over $c,d\in C_{t_\lambda}$, one sets $e=c^{-1}d$. For fixed $c$, right multiplication $\pi\mapsto \pi c$ is a bijection of $S_n$, so the remaining sum over $c$ contributes the factor $|C_{t_\lambda}|$. The same simplification is used for the $S_{n-1}$-blocks $Y_\mu$. 
\end{remark}

\subsection{The sparse certificate}
After symmetry reduction, the $n=8$ example can be heuristically further reduced. The sparsification step is used to identify a small support. Once this support is fixed, the final certificate is obtained and verified independently over $\mathbb Q$. 

In our numerical experiments, removing any one of the blocks $Q_{(8)}$, $Q_{(5,3)}$, $Q_{(6,2)}$, or $R_{(7)}$ made the SDP either infeasible or gave an objective value $> 0.0125$. Here the block indices are partitions of $8$ for the matrix $Q$, and of $7$ for the matrix $R$, hence correspond to irreducible representations of $S_8$ respectively $S_7$. If we set all other blocks to zero but keep only these blocks free, the objective value $0.0107\ldots$ was the same as the value from the full SDP (up to 20 decimals calculated with \texttt{SDPA-GMP}). So, the SDP could be restricted to those four blocks of sizes~$12+3+8+14$. 

To find a further reduction (heuristically) we set diagonal elements of these respective blocks to zero and checked whether the objective remains the same. Note that the blocks~$Q_{(8)}$ and~$R_{(7)}$ correspond to the~$S_8$-, respectively~$S_7$-invariant polynomials. So in order to find a certificate in which most of the polynomials which are squared are themselves symmetric, we first tried to set as many diagonal elements of $Q_{(6,2)}$ and  $Q_{(5,3)}$ to zero. We do this by checking all possible subsets of the 8+3 rows of the two matrices combined. It is found that we can set~$6$ of the diagonal elements to zero, without changing the objective value. We set these diagonal elements to zero, and then try to set as many diagonal elements in~$Q_{(8)}$ and~$R_{(7)}$ to zero without changing the objective value. This leads us to a relatively small set of representative basis elements used in the final certificate. The used basis elements (corresponding to nonzero diagonal elements hence nonzero rows/columns) are given in Table~\ref{table:certificate}.

\begin{table}[H]
\centering
\begin{tabularx}{\textwidth}{|l|X|}
\hline
{Matrix block} & {Representative basis polynomials used in the certificate} \\
\hline
$Q_{(8)}$ &
$1$, $\sum_{i=1}^8 x_i$, 
  $\sum_{1 \leq i < j \leq 8} x_i x_j$, 
$\sum_{1 \leq i < j \leq 8} (x_i^2 x_j + x_i x_j^2)$, $\sum_{1 \leq i < j < k \leq 8} x_i x_j x_k$,
$\sum_{i=1}^8 x_i^4$, $\sum_{1 \leq i < j \leq 8} (x_i^3 x_j + x_i x_j^3)$, $\sum_{1 \leq i < j \leq 8} x_i^2 x_j^2$,
$\sum_{1 \leq i < j < k \leq 8} (x_i^2 x_j x_k + x_i x_j^2 x_k + x_i x_j x_k^2)$,
$\sum_{1 \leq i < j < k < \ell \leq 8} x_i x_j x_k x_\ell$ \\
\hline
$Q_{(6,2)}$ & $x_7 x_8 - x_2 x_7 - x_1 x_8 + x_1 x_2$,
$(x_7 x_8 - x_2 x_7 - x_1 x_8 + x_1 x_2) \sum_{i=3}^6 x_i$,
$(x_7 x_8 - x_2 x_7 - x_1 x_8 + x_1 x_2) \sum_{3 \leq i < j \leq 6} x_i x_j$ \\
\hline
$Q_{(5,3)}$ &$x_6 x_7 x_8 - x_3 x_6 x_7 - x_2 x_6 x_8 + x_2 x_3 x_6 - x_1 x_7 x_8 + x_1 x_3 x_7 + x_1 x_2 x_8 - x_1 x_2 x_3$,
$(x_6 x_7 x_8 - x_3 x_6 x_7 - x_2 x_6 x_8 + x_2 x_3 x_6 - x_1 x_7 x_8 + x_1 x_3 x_7 + x_1 x_2 x_8 - x_1 x_2 x_3)\sum_{i=4}^5 x_i$ \\
\hline
$R_{(7)}$ & $1$, $\sum_{i=1}^7 x_i$, $\sum_{i=1}^7 x_i^2$, $\sum_{1 \leq i < j \leq 7} x_i x_j$, $\sum_{i=1}^7 x_i^3$,
$\sum_{1 \leq i < j \leq 7} (x_i^2 x_j + x_i x_j^2)$, $\sum_{1 \leq i < j < k \leq 7} x_i x_j x_k$,
$x_8$, 
$x_8 \sum_{i=1}^7 x_i^2$, $x_8 \sum_{1 \leq i < j \leq 7} x_i x_j$,
$x_8^2$, $x_8^2 \sum_{i=1}^7 x_i$, $x_8^3$ \\
\hline
\end{tabularx}
\caption{Representative basis polynomials used per matrix in the certificate (unused basis elements are omitted). } \label{table:certificate}
\end{table}

\section{Exact rationalization}\label{sec:rational}
To round our numerical certificate to a rational certificate, we solved the symmetry-reduced SDP numerically using only the blocks $Q_{(6,2)}$, $Q_{(8)}$, $Q_{(5,3)}$ for the matrix~$Q$, together with the block $R_{(7)}$ for $R$. Furthermore, we imposed the constraints that the diagonal elements corresponding to unused basis elements are set to zero, so that only the basis elements in Table~\ref{table:certificate} are used. To obtain a suitable point for rationalization, we fixed $ C=\tfrac{11}{1000}$ and maximized $t$ subject to the coefficient-matching constraints and the lower-bound constraints
$$
Q_{(6,2)}^{\mathrm{act}}\succeq tI,\qquad Q_{(8)}^{\mathrm{act}}\succeq tI,\qquad Q_{(5,3)}^{\mathrm{act}}\succeq tI,\qquad R_{(7)}^{\mathrm{act}}\succeq tI.
$$
Here the superscript $\mathrm{act}$ denotes the active principal submatrix obtained after deleting the rows and columns forced to be zero. Numerically, this produced a feasible point with $t \approx 10^{-12}$. This numerical solution was used as a starting point for rationalization. The final feasibility and positive semidefiniteness checks below are exact. Let $z_0$ be a vector containing all Gram-matrix entries of the high-precision numerical solution. Exact coefficient matching of monomials gives a rational affine system
$$
Az=b,\qquad A,b\in\mathbb Q.
$$
We rounded $z_0$ to a rational vector $z_{\mathrm{round}}$ and computed the exact rational residual
$$
r=b-Az_{\mathrm{round}}.
$$
We verified by exact Gaussian elimination over $\mathbb Q$ that $A$ has full row rank,\footnote{In our case~$A \in \Q^{67 \times 155}$ in the final system, of rank~$67$. The blocks have sizes $10,3,2,13$, respectively, so the system has $10\cdot 11/2+3\cdot4/2+2\cdot3/2+13\cdot14/2=155$ variables.} so $AA^{\mathsf T}$ is invertible. The corrected vector
$$
z=z_{\mathrm{round}}+A^{\mathsf T} (AA^{\mathsf T})^{-1} r
$$ 
then satisfies $Az=b$ exactly over $\mathbb Q$. We reconstructed the rational Gram matrices from $z$, and hence constructed an exact rational certificate. Positive semidefiniteness was verified exactly by checking all principal minors of the rational Gram matrices. The determinants were computed over~$\mathbb Q$:  all principal minors were nonnegative. We also verified by exact coefficient matching that the associated polynomial is precisely $x_1x_2\cdots x_8+\frac{11}{1000}$. The ancillary files contain:
\begin{enumerate}
\item rational Gram matrices for the four nonzero blocks;
\item a Julia script verifying exact coefficient matching over $\mathbb Q$;
\item a Julia script verifying positive semidefiniteness over $\mathbb Q$.
\end{enumerate}
This proves Theorem~\ref{thm:conjecturecounterexample}. The exact rational certificate and verification scripts are available at \href{https://doi.org/10.5281/zenodo.20430334}{\texttt{doi:10.5281/zenodo.20430334}} and are included as ancillary files with the arXiv submission.\footnote{Additional files to compute the numerical SDP values are available in the GitHub repository \url{https://github.com/svenpolak/HypercubeSOS/}.}

\section{Further remarks: a Boolean quotient relaxation}\label{sec:ideal}

The computations suggest that, at the diagonal order $2d=n$, the quadratic-module SDP may, for this family of polynomials $x_1\cdots x_n$, behave like the corresponding Boolean quotient relaxation. The observations here are not used in the proof of Theorem~\ref{thm:conjecturecounterexample}.

\begin{remark}
In the numerical SDP solutions for certificates of $x_1\cdots x_n+C_n\in M_{n,n}(\mathbf g)$, for $n=2,4,6,8,10,12$, we observed a consistent pattern. The part corresponding to $\rho$ could be restricted to the trivial $S_{n-1}$-block $R_{(n-1)}$, and the $S_n$-invariant sum-of-squares part could be restricted to the two-row blocks
$$
Q_{(n)},\ Q_{(n-1,1)},\ Q_{(n-2,2)},\ldots,Q_{(n/2,n/2)}.
$$
All other blocks could be set to zero without changing the numerical optimum. We do not have a proof of this block restriction for the quadratic-module SDP.
\end{remark}

This observation suggests comparing the quadratic-module SDP with a related
Boolean quotient relaxation, where the appearance of only one- or two-row blocks is automatic. Consider the Boolean ideal $I_n=\langle x_1^2-x_1,\ldots,x_n^2-x_n\rangle$. Let $C_{n,2d}^{\mathrm{bool}}$ be the smallest constant $C$ such that $x_1\cdots x_n+C\equiv \sum_j q_j^2 \pmod{I_n},\,\,\deg q_j\le d$. The Boolean quotient relaxation is equivalent to the following preordering relaxation. Write $g_S:=\prod_{i\in S}g_i$, where $g_i=x_i-x_i^2$. The degree-$2d$ truncated preordering $\mathcal T_{n,2d}(\mathbf g)$ generated by the $g_i$'s consists of all polynomials of the form $\sum_{S\subseteq[n]}\sigma_S g_S$, where the $\sigma_S$'s are sums of squares and $\deg(\sigma_Sg_S)\le 2d$. Let $C_{n,2d}^{\mathrm{pre}}$ be the corresponding smallest constant $C$ such that $x_1\cdots x_n+C\in \mathcal T_{n,2d}(\mathbf g).$

\begin{lemma}
For all $n$ and $d$,
$$
C_{n,2d}^{\mathrm{bool}} = C_{n,2d}^{\mathrm{pre}} \le C_{n,2d}.
$$
\end{lemma}

\begin{proof}
The inequality follows because the quadratic module is contained in the preordering. For the equality, first reduce a preordering certificate modulo $I_n$. All terms with $S\neq\emptyset$ vanish, since $g_i\equiv0\pmod{I_n}$, which gives a Boolean sum-of-squares certificate.

Conversely, every Boolean square lifts to the preordering. We may assume that $q$ is squarefree. Then
\begin{align}\label{qsquaredI}
\mathrm{red}(q^2)=\sum_{S\subseteq[n]} g_S(\partial_S q)^2,
\end{align}
where $\mathrm{red}$ denotes the reduction modulo~$I_n$ and where $\partial_S x_A=x_{A\setminus S}$ if $S\subseteq A$, and $\partial_S x_A=0$ otherwise. Equation~\eqref{qsquaredI} follows by bilinearity from the monomial case: for monomials $x_A,x_B$, we have $x_{A\cup B} = \sum_{S\subseteq A\cap B}g_S x_{A\setminus S}x_{B\setminus S}$. Moreover, if $\deg q\le d$, then each term $g_S(\partial_Sq)^2$ has degree at most $2d$. Hence every Boolean sum-of-squares certificate gives a degree-$2d$ preordering certificate. 
\end{proof}

The Boolean quotient SDP is significantly smaller after symmetry reduction than the original quadratic module SDP. Since all monomials are squarefree modulo $I_n$, only the exponent patterns $\xi=(0<1;(n-k,k))$ occur in the representative-vector construction of Section~\ref{subsec:repr-sn}. Hence only one- or two-row partitions appear. The dual can be averaged over $S_n$ and  written in the simple form
$$
C_{n,2d}^{\mathrm{bool}}= \max_{y_0,\ldots,y_n\in\mathbb R}\left\{-y_n:\ y_0=1,\ \left(y_{|A\cup B|}\right)_{\substack{A,B\subseteq[n]\\|A|\le d,\ |B|\le d}}\succeq 0\right\}.
$$
For smaller~$n$, we computed the value directly from this $n+1$-variable moment matrix formulation. For larger~$n$, we used the $S_n$-block diagonalization of the moment matrix.  Numerically, for $n=2,4,6,8,10,12$, the Boolean/preordering value $C_{n,n}^{\mathrm{bool}}$ agrees with the quadratic-module value $C_n=C_{n,n}$.\footnote{Using symmetry, the Boolean SDP could be solved numerically for significantly larger~$n$.} Hence the question:

\begin{problem}
Is it true that, for every even $n$, one has
$$
C_{n,n}^{\mathrm{bool}}=C_{n}?
$$
Equivalently, for the polynomial $x_1\cdots x_n+C$, do the degree-$n$
preordering threshold and the degree-$n$ quadratic-module threshold on the
hypercube coincide?
\end{problem}
A positive answer would mean that the De Klerk--Laurent correction constant $C_{n}$ is determined by the Boolean quotient. This would be a special feature of the symmetric monomial $x_1\cdots x_n$, not a general equality between the truncated preordering and the truncated quadratic module.
\section*{Acknowledgements}
The author is grateful to Monique Laurent for sharing the conjecture, for suggesting the use of symmetry reduction, and for valuable discussions. The author also thanks Etienne de Klerk for helpful comments on the motivation of the conjecture.

\selectlanguage{english} 

\end{document}